\documentclass[11pt]{amsart}
\usepackage{amsthm, amsmath}
\usepackage{amssymb,amsfonts}
\usepackage{mathtools}
\usepackage{filecontents}\usepackage[margin=1.2in]{geometry}
\usepackage{enumerate}
\usepackage{mathrsfs}
\usepackage{mathtools}
\usepackage{xcolor}

\usepackage[alphabetic]{amsrefs}

\usepackage{amsthm}
\theoremstyle{definition}

\newtheorem{thm}{Theorem}[section]
\theoremstyle{definition}

\newtheorem{cor}[thm]{Corollary}
\newtheorem{prop}[thm]{Proposition}
\newtheorem{lem}[thm]{Lemma}

\newtheorem{defn}[thm]{Definition}

\newtheorem{term}[thm]{Terminology}

\newtheorem{rem}{Remark}[section] 

\def\C{{\mathbb C}}
\def\Z{{\mathbb Z}}
\def\D{{\mathbb D}}

\def\Q{{\mathbb Q}}

\def\F{{\mathbb F}}

\def\cO{{\mathcal O}}

\DeclareMathOperator{\End}{End}

\DeclareMathOperator{\Hom}{Hom}

\DeclareMathOperator{\sups}{ss}
\DeclareMathOperator{\tor}{tor}
\DeclareMathOperator{\et}{et}

\def\G{{\mathcal G}}
\def\F{{\mathbb F}}
\def\G{{\mathscr G}}
\def\F{{\mathbb F}}

\def\p{{\mathfrak p}}

\DeclareMathOperator{\ord}{ord}
\newcommand\isomto{\stackrel{\sim}{\smash{\longrightarrow}\rule{0pt}{0.4ex}}}

\usepackage{tikz-cd}

\usepackage {hyperref}
\hypersetup{colorlinks=true,
linkcolor=blue
}

\begin{document}
\title[]{Abelian varieties with  Real multiplication\,: \\ classification and isogeny classes over finite fields}
\author{Tejasi Bhatnagar, Yu Fu}
\maketitle
\vspace{-1cm}
\begin{abstract}
In this paper, we provide a classification of certain points on Hilbert modular varieties over finite fields under a mild assumption on Newton polygon. Furthermore, we use this characterization to prove estimates for the size of isogeny classes. 

\end{abstract}

\tableofcontents
\section{Introduction}

\noindent Throughout this paper, $\F_q$ will denote the finite field with $q$ elements where $q$ is a power of an odd prime $p$. We denote by $k$ the algebraic closure of $\F_q$. In \cite{De69}, Deligne provides a classification of ordinary abelian varieties over finite fields in terms of certain $\Z$-modules along with an action of a ``Frobenius map". A crucial step in this classification requires the existence of the ``Serre-Tate canonical lift" of an ordinary abelian variety to characteristic zero. Using the ``canonical lift" of an ordinary abelian variety, Deligne provides a characterisation of such abelian varieties by associating to them the integral homology of their lift. A similar classification for simple almost ordinary abelian varieties over finite fields is proved in \cite{OS20}. In order to extend Deligne's work, Oswal and Shankar construct canonical lift(s) of almost ordinary abelian varieties using Grothendieck-Messing theory. Moreover, the lifts are characterized by the property that \textit{all} of the endomorphisms lift to characteristic zero. Using these canonical lift(s),  they classify simple almost ordinary abelian varieties, and as an application, they give a lower bound of the size of almost-ordinary isogeny classes. This paper  aims to generalise their work to certain abelian varieties, (not necessarily simple) with \textit{real multiplication} over finite fields with some assumptions on $p$, as described below. (See Theorem \ref{mainthm})

\subsubsection{The Hilbert moduli space}\noindent Let $L$ be a totally real number field of degree $g$ over $\Q$. Let $\cO_L$ be its ring of integers. We denote by $\mathcal{H}_{L,\mathfrak{a}}$ the coarse Hilbert moduli space over $\F_q$ associated to $L$. This moduli space parametrises abelian varieties with real multiplication with additional data. More precisely, we describe the moduli problem as follows.
\begin{defn}
Fix a fractional ideal $\mathfrak{a}$ of $L$. Define $\mathcal{H}_{L,\mathfrak{a}}$ to be the functor such that for any $\F_{q}$-scheme $S$, $\mathcal{H}_{L, \mathfrak{a}}(S)$ is the set of isomorphism classes of triples,$(A, \iota, \lambda)$ where:

\begin{itemize}
\item[(1)] $A \to S$ is a $g$-dimensional abelian scheme over $S$;
\item[(2)] $\iota: \cO_L\hookrightarrow \End A$ is an embedding into the endomorphism ring;
\item[(3)] $\lambda:\mathfrak{a} \isomto \operatorname{Hom}_{\cO_{L}, S}^{\operatorname{Sym}}(A, A^{\vee}) $ is an isomorphism of $\cO_{L}$ modules. Here $\lambda$ identifies the set of polarizations with the set of totally positive element $\mathfrak{a}^{+} \subset \mathfrak{a}$. For every $a \in \mathfrak{a}^{+}$, $\lambda(a)$ is an $\mathcal{O}_{L}$-linear polarization of $A$, and the homomorphism $A \otimes_{\mathcal{O}_{L}} \mathfrak{a} \isomto  A^{\vee}$ induced by $\lambda$ is an isomorphism of abelian schemes.
\end{itemize} 

\end{defn}
\noindent See \cite[Chapter 3.6.1]{Goren} for a complete description of the moduli problem. For the purpose of this paper, we will work with the coarse moduli space and classify $\mathfrak{a}$-polarised abelian varieties with RM, for a fixed fractional ideal $\mathfrak{a}$ of $L.$ In particular, when the polarisations are principal, the corresponding fractional ideal class is the inverse different of $L$, denoted by $\mathcal{D}_{L/\Q}^{-1}$. We restrict to this case while counting the size of isogeny classes.

\begin{defn}
 Two abelian varieties $(A,\lambda_A, \iota_A)$ and $(B,\lambda_B, \iota_B)$ in $\mathcal{H}_{L,\mathfrak{a}}(\F_q)$ are said to be  \textit{isogenous} if there is an isogeny $\varphi: A\rightarrow B$ over $\overline\F_q$ that is compatible with the action of $\cO_{L}$, that is $\varphi \circ \iota_A = \iota_B \circ \varphi$. Moreover, the isogeny preserves polarization up to  scaling by $L^{\times}$. 
 \end{defn}
 \begin{rem}
From this section onwards we fix a totally real field $L$ of degree $g$ over $\Q$ and a fractional ideal $\mathfrak{a}$ of $L$. When we say that an abelian variety $A$ over $\F_q$ has RM, we mean that $A$ over $\F_q$ is $\mathfrak{a}$-polarised and has RM by $\mathcal{O}_L.$ That is, $A$ is an $\F_q$-point of $\mathcal{H}_{L,\mathfrak{a}}.$
\end{rem}
\subsubsection{The main theorem: classification }\noindent Let $A$ be a geometrically simple $\F_q$-point of $\mathcal{H}_{L,\mathfrak{a}}$.  We assume that the $p$-rank\footnote{We recall that the \textit{$p$-rank} of an abelian variety is defined to be the integer $f=\operatorname{dim}_{\F_{p}}\Hom(\mu_{p}, A[p])$.} of $A$ is $g-a$ for some $0\leq a \leq g-1.$ Furthermore, we assume that $p$ is totally split in $\cO_L$. This assumption guarantees that the $p$-divisible group splits into either ordinary or local-local factors (see Lemma \ref{ordll}), implying that the endomorphism algebra of $A$ is, in fact, a CM field $K.$ Consequently, the extension $K\otimes \Q_p$, splits as $K_{\ord}\oplus K_{\sups}$ where $K_{\ord}$ is a direct sum of $2(g-a)$ factors of $\Q_p$, while $K_{\sups}$ splits as a sum of $a$ quadratic extensions of $\Q_p$. We write $K_{\sups} =  \oplus_{1\leq i\leq a} K^i_{\sups}$. 

\begin{term} We call the abelian variety $A$ (or its isogeny class) \textit{totally ramified}, if all the factors of $K_{\sups}$ are ramified, \textit{inert} if all the factors are inert, and \textit{ramified} if some of the factors are inert and some are ramified. We note that the above terminology extends to non-simple abelian varieties as well, and this result is true in generality, thus giving us a classification for RM abelian varieties with non-commutative endomorphism ring as well.

\end{term}

\noindent  In what follows, we denote by $\mathcal{C}_h$ the category of abelian varieties over $\F_q$ with RM such that the characteristic polynomial of their relative Frobenius is $h$. Its objects are $\F_q$-points of $\mathcal{H}_{L,\mathfrak{a}}$, see \ref{catdef} for a precise definition. The morphisms in this category are morphisms of abelian varieities over $\F_q$ that are compatible with the action of $\mathcal{O}_L$ and preserve polarisations up to scaling by $L^{\times}$. We classify the objects of $\mathcal{C}_h$ in terms of certain $\Z$-modules with a ``Frobenius" map and an action by $\mathcal{O}_L$. These objects along with the defined morphisms form a category of \textit{Deligne modules with RM} which we  denote by $\mathcal{L}_h.$ See \ref{dmodrm} for the precise definition. We prove the following classification result in this paper.

\begin{thm}\label{mainthm}
Let $\mathcal{C}_h$ and $\mathcal{L}_h$ be defined as above. We assume that the $p$-rank of abelian varieties in $\mathcal{C}_h$ equals to $g-a$ for some $0\leq a\leq g-1.$  

\begin{enumerate}
    \item If $\mathcal{C}_h$ is totally ramified, then there exists  $2^a$ canonical functors $\mathcal{F}_i: \mathcal{C}_h\rightarrow \mathcal{L}_h$ with $1\leq i\leq 2^a.$ Each of these functors induce equivalences between two categories.  
    \item If $\mathcal{C}_h$ is inert, then there are $2^a$ full subcategories $\mathcal{C}_{h,i}$ of $\mathcal{C}_{h}$ where $1\leq i\leq 2^a.$ The functors $\mathcal{F}_i$ from each of these subcategories to $\mathcal{L}_h$ induce an equivalence between the categories. 
\item If $\mathcal{C}_h$ is ramified,  analogous to $(2)$, then there are $2^{a-k}$ subcategories of $\mathcal{C}_{h}$ where $k$ is the number of inert factors of $K_{\sups}.$ The functors $\mathcal{F}_i$ from each of these subcategories to $\mathcal{L}_h$ for $1\leq i\leq 2^{a-k}$, induce an equivalence between the categories. 
\end{enumerate}

\end{thm}
\subsection{Strategy to prove the classification theorem.}  We first prove that $A$ lifts to characteristic zero along with all its endomorphism algebra. More precisely, when $A$ is totally ramified, it has exactly $2^a$ lifts to characteristic zero. When $A$ is inert, then it has one canonical lift, and finally, in the case when $A$ is ramified, its number of lifts depend on the number of ramified quadratic extensions in the direct sum $K_{\sups}$. Fix a lift $\widetilde{A}$ of $A$, we associate it with the integral homology $H_1(\widetilde{A}\otimes_{i} \C, \Z)$. The Frobenius map on the homology lattice comes from the lift of the Frobenius of $A.$ This gives us a functor from the category of RM-abelian varieties to the category of $\Z$-modules with a ``Frobenius" map. 
When all is said and done, however, there is no unique choice of a functor between the categories $\mathcal{C}_h$ and $\mathcal{L}_h$. As in \cite{OS20}, this ambiguity comes from possible CM-types on $K.$ Suppose $\sigma_1,\sigma^{\prime}_1, \sigma_2,\sigma^{\prime}_2,\dots \sigma_a,\sigma^{\prime}_a$ are the embeddings corresponding to the slope $1/2$ part of $a,$ then to give a possible CM type on $K$, we could choose either one of $\sigma_i$ or $\sigma_i^\prime$ for $1\leq i\leq a$ together with the embeddings corresponding to the slope $0$ part.
\noindent In each of the cases, this relates to the functoriality of the association between the two categories. For example, in the totally ramified case, we get a functorial map between $\mathcal{C}_h$ and 
$\mathcal{L}_h$ once we choose a particular lift of $A$ out of $2^a$ possibilities. On the other hand, in the inert case, even though we have one canonical lift of $A,$ the association of $A$ with the integral homology of its lift is not functorial. Once we restrict to one of $2^a$ subcategories of $\mathcal{C}_h$ described in section \ref{3}, we get a functorial association. 
\vspace{2mm}

\subsection{Application to the size of isogeny classes.} The classification point of view of abelian varieties over finite fields has been utilised in many directions, many computational in nature. For example, Howe in \cite{howe} studies polarisation  building on Deligne's work to show that there exists a principally polarised abelian variety in certain isogeny classes. In \cite{ST18}  Shankar and Tsimerman
 formulate and give evidence of certain conjectures regarding intersections of irreducible varieties and isogeny classes  in $\mathcal{A}_g$, the moduli space of principally polarised abelian varieties. One of the applications of classification results  to answer such questions is to estimate size of isogeny classes. 
 
\noindent In this section, using the characterization in Theorem \ref{mainthm}, we estimate the size of isogeny classes of a $\F_{q}$ point, $A \in \mathcal{H}_{L,\mathfrak{a}}(\F_{q})$ of the Hilbert modular variety. We define $I(A, q^n)$ to be the set of abelian varieties in $\mathcal{H}_{L,\mathfrak{a}}(\F_{q^{n}})$ such that there exists an isogeny to $A$ over $\overline\F_{q}$. Denote by $N(A, q^{n})$  the size of the set $I(A, q^n)$.  We prove the following estimate for $N(A, q^{n})$.
\begin{thm}\label{isogsize}
For all but finitely many $n$, we have 
$N(A, q^{n}) = q^{\frac{n}{2}(g-a+o(1))}$.
\end{thm}
We note that we get the leading term for all but finitely many $n$, while the error is an asymptotic as $n$ goes to infinity. 

\noindent In \cite {AC02}, Achter and Cunningham prove a precise formula for the size of isogeny classes of ordinary abelian varieties with RM. Their paper uses a different technique of orbital integrals. The estimate in \cite[Theorem 3.1]{AC02} agrees with our Theorem \ref{isogsize}, when the $p$-rank of the abelian variety is $g$, that is, for ordinary abelian varieties with RM.

\subsection{Strategy to prove the bounds for the size of the isogeny classes.}
 In order to find the lower bound for $N(A, q^{n})$, we count the number of  points in $\mathcal{H}_{L,\mathfrak{a}}(\F_{q^{n}})$ with endomorphism by the smallest order $R_n\subset   \mathcal{O}_{K}$ that occurs as an endomorphism ring in the isogeny class of $A.$  We characterize this ring in proposition \ref{endcom}. The classification theorem gives us a bijection between $\mathcal{C}\ell(R_n)$ and the  $\F_{q^n}$ points  of $\mathcal{H}_{L,\mathfrak{a}}$ with endomorphism ring $R_n$. Finding a lower bound therefore reduces to estimating the size of the class group. Furthermore, in section \ref{upperbound}, we prove that this bound is sharp.  We compute the upper bound by estimating size of the class group of \textit{all} over orders of $R_n.$ We note that $R_n$ and all its over orders are Gorenstein, a property that gives a unique characterization of its over-orders in terms of the subgroups of $\cO_K/R_n$, a cyclic group in our case (see Proposition \ref{upbound}). 


\noindent\subsection{Organization of the paper} Our first step is to construct canonical lift(s) of abelian varieties with RM in section $2$. We then give the main theorem of classification in section $3$, and finally, section $4$ provides estimates for the size of isogeny classes.

\section{Acknowledgements.} 
\noindent  Our approach owes a substantial intellectual debt to the work of Oswal and Shankar. We are very grateful to Ananth Shankar for introducing this question to us and for his help throughout the project. We wish to thank Nathan Kaplan and Ziquan Yang for helpful conversations. We thank Jordan Ellenberg for many useful comments on the introduction. We were partially supported by NSF grant DMS-2100436.
\section{Canonical lift}
\subsection{The simple case.} For this section, we fix a geometrically simple $g$-dimensional abelian variety $A$ over $\F_q$ that has RM by $\cO_L.$ Let $g-a$ be the $p$-rank of $A$ where $1\leq a\leq g-1.$ We assume that $p$ is totally split in $\cO_L.$  Let $A[p^{\infty}]$ be the $p$-divisible group of $A.$ 
Let $\End^{0}(A)$ be the endomorphism algebra of $A$ defined over $\F_q$. Let $K=\Q(\pi)$ be the field generated by the relative Frobenius of $A$ over $\F_q$. We denote by $W(\F_q)$, the ring of Witt vectors over $\F_q$ and let $W=W(k).$

\noindent Our first step is to prove that the endomorphism algebra of $A$ is commutative. To that end, we state a useful lemma that describes the decomposition of $A[p^{\infty}]$ over $k$ based on the splitting of $p.$
   
\begin{lem}\label{ordll}\cite[Lemma 5.2.1]{AG04} Let $x\in \mathcal{H}_{L,\mathfrak{a}}(k)$. Then its $p$-divisible group
 $\G_x$ decomposes as product of $p$-divisible groups with factors $\G_{x,\p}$ corresponding to the primes lying above $p$ in $\cO_L$. The  Dieudonn\'e module of $\G_{x,\p}$  is an $\cO_{L_{\p}}\otimes W(k)$ module where $\cO_{L_\p}$ is the completion of $\cO_L$ at $\mathfrak{p}.$
Corresponding to each prime $\p$, the associated Dieudonn\'e module is either ordinary or local-local. Moreover, the Dieudonn\'e module of each local-local factor is free of rank two. 
\end{lem}

\noindent Let $R$ and $S$ be the endomorphism rings of $A$ and $A[p^{\infty}]$ respectively over $\F_q.$ We know by Tate's theorem that $R\otimes \Z_p =S$. 
\begin{prop}\label{endcom}
The endomorphism ring $R$ and (hence) the $\Z_p$-algebra $S$ is commutative.  Moreover, $S$ admits a decomposition over $\F_q$ into $S_{\text{\'et}}\oplus S_{\tor}\oplus S^i_{\sups}$ for $1\leq i\leq a$ where $S^i_{\sups}$ is rank two $\Z_p$-algebra. Further, for each such $ i,$ the algebra $S^i_{\sups}$ is maximal in its field of fractions.
\end{prop}
\begin{proof}

 We show that $\End^0(A) = K$. Since the endomorphisms by $\cO_L$ are defined over $\F_q$, we get a splitting of the $p$-divisible group and hence of its Dieuodonn\'e module over $\F_q$ into rank two modules.  By Lemma \ref{ordll}, we get a decomposition into  ordinary and supersingular factors $\mathbb{D}_{\ord}\oplus \mathbb{D}_{\sups}$ where $\mathbb{D}_{\sups}$ is a product of $a$ supersingular rank. Since $p$ is completely split in $\cO_L$, we write $\cO_L\otimes \Z_p = \Z_p^{g-a}\times \Z_p^a$, where the $\Z_p^a$ factor acts on $\mathbb{D_{\sups}}$. Choose a basis of idempotents $e_1, e_2, \dots ,e_a$ of $\Z_p^a$. Since $e_1,e_2, \dots, e_a$ are idempotent elements, we see that $\D_{\sups}$ is isomorphic to the direct sum  $e_1\cdot \D_{\sups}\oplus\cdots \oplus e_a\cdot \D_{\sups}$ such that the action of $\Z_p^a$ respects the direct sum. Let $f(x)$ denote the characteristic polynomial of $\pi$ over $\Z$. Since $A$ is geometrically simple, we have that $f(x) = r(x)^{m}$ where $r(x)$ is an irreducible polynomial over $\Q$. We have a decomposition of $f(x) = f_{\text{ord}}(x)f_{\sups}(x)$ over $\Z_p$ corresponding to the decomposition of the Dieudonn\'e module. We note that
$f_{\sups}$ is product of factors $f^i_{\sups}$ for $1\leq i\leq a$. We have a similar decomposition of $r(x)$ as well. We know that the exponent $m$ is either $1$ or $2$ since each $r^i_{\sups}$ is degree $2.$ Suppose that $f(x) = r(x)^2.$ Then each $r^i_{\sups}$ must be a degree $1$ polynomial. Let $\alpha$ be the root of say, $r^1_{\sups}(x).$ Then, $f^1_{\sups}(x) = (x-\alpha)^2$. However, this implies that the minimal polynomial is reducible over $\Q$, that is, $A_{\F_{q^2}}$ is not simple, which is a contradiction. This proves that $\End^0(A)$ equals its center. Therefore, $L$ is the degree $g$ totally real field of $K.$ Let $\iota$ denote the complex conjugation of $K$ over $L.$ Consider the $g-a$ ordinary primes in $\cO_L$ over $p.$ These further split into two different primes $\p$ and $\p^{\prime}$ in $K$ such that $\p^{\prime} = \iota{\p}$. On the other hand, the supersingular primes either stay inert or ramified as they correspond to a degree $2$ extension of $\Q_p.$
This describes the Newton polygon of $A$ and hence gives a decomposition of the $p$-divisible group over $\F_q$ into $\G_{\text{\'et}}\times \G_{\tor}\times \G_{\sups}$. Therefore, we get the corresponding decomposition of $S$ as well. 
Finally, in order to show that $S^{i}_{\sups}$ is maximal, we follow a similar argument given at the end of \cite{OS20}. Without loss of any generality, we do this for $i=1.$ Suppose $\G^{\prime}$ is a $p$-divisible group isogenous to $\G^1_{\sups}$ such that $\End(\G^{\prime})$ is maximal.  Let $j:\G_0\rightarrow \G^1_{\sups}$ be any isogeny. We will show that $\End(\G')$ preserves $\ker j,$ proving that $\End(\G')\subset \End(\G^1_{\sups}).$ Note that  $\G_{0}$ is a one dimensional $p$-divisible group and so for each $r$, it has a unique subgroup of order $p^r$ corresponding to the unique quotient of order $p^r$ of its Dieudonn\'e module. Therefore, $\End(\G^1_{\sups})$ preserves  the kernel of $j$ which is a unique subgroup of order $p^r$ in $\G_0$.
\end{proof}
\noindent Let $W^{\prime}$ denote the rings of integers of a \textit{slightly ramified} extension  $W^{\prime}(1/p)$ of $W(1/p)$. That is, $[W^{\prime}(1/p): W(1/p)]$ is at most $p-1.$ 
The following proposition in \cite{OS20} uses Grothendieck-Messing theory to construct lifts of $A$ to $W^{\prime}$. 
\begin{lem}\label{AAlift}
Let $\G_{\sups}$ denote a one dimensional  $p$-divisible group over $\overline{\F}_p.$ Suppose $\cO\subset \End(\G_{\sups})$ is an integrally closed rank two $\Z_p$-algebra. Then

\begin{enumerate}
    \item If $\cO$ is unramified, there exists a unique lift $\G_{\sups}$ to $W$ such that the action of $\cO$ lifts. 
    \item If $\cO$ is ramified, then there exists two lifts of $\G_{\sups}$ to $W^{\prime}$ such that the action of $\cO$ lifts.
\end{enumerate}
\end{lem}

\noindent The proof of this proposition is given in \cite{OS20}. Oswal and Shankar construct lifts of $A$ to $W^{\prime}$ by lifting its $p$-divisible  group. More precisely, Grothendieck-Messing theory gives a bijection between the lifts of $\G_{\sups}$ to $W^{\prime}$ with rank-one filtrations of the  Dieudonn\'e module $\D\otimes_{W} W^{\prime}$ that reduce to the kernel of the Frobenius on $\D$ modulo $p.$ Moreover, in order to lift the endomorphisms, we need the filtration to be invariant under the action of the endomorphism algebra $\cO$ of $\G_{\sups}$. We can explicitly compute such filtrations depending on whether $\cO$ is inert or ramified. We either have a unique choice of such a filtration or two such choices that correspond to the lift(s) of $\G_{\sups}$ to $W^{\prime}$.

Therefore, in the geometrically simple case, Lemma \ref{AAlift} guarantees the following theorem.
\begin{thm}
We assume the notation described at the beginning of this section. The super-singular part $A[p^{\infty}]_{\sups}$ of the $p$-divisible group admits at-most
 $2^a$ lifts to $W^{\prime}.$ Moreover, the lifts are characterized by the property that all of the endomorphisms of $A[p^{\infty}]_{\sups}$ lift as well.
\end{thm}





\subsection{The non-simple case} In this section, we extend the lifting theorem to the non-simple case. As in the simple case, it is enough to establish the result for the supersingular part of the $p$-divisible group.
\begin{lem}
Suppose $A$ is a non-simple abelian variety of dimension $2g$ with RM by $\cO_L$ over $\F_q$. Then $A$ is isogenous to $B^m$, an $m$-fold product of a simple RM abelian variety $B$ defined over a possible field extension of $\F_q$. 
\end{lem}

\begin{proof}
We show that if $A = B_1\times B_2$, a product of two geometrically simple RM abelian varieties then $B_1$ is isogenous to $B_2$. Let $g_1$ be the dimension of $B_1$ and $g_2$ be the dimension of $B_2$. Without any loss of generality, we assume $g_1\geq g_2.$ Then by the definition of RM abelian varieties, $B_1$ has RM by a totally real field, say $L_1$, of degree $g_1$, while $B_2$ has RM by a totally real field $L_2$ of degree $g_2.$ Now, if $B_1$ were not isogenous to $B_2$, then $\End^0 A = \End^0 (B_1\times B_2) = \End^0(B_1)\times \End^0(B_2)$. We note that $L$ is a degree $g$ field in $\End^0 A$. However, $\End^0(B_1)\times \End^0(B_2)$, a direct sum of CM fields can at most have dimension  $g_1$ totally real field embedded in it, which is a contradiction.
\end{proof}

\begin{prop}
Let $A$ be a non-simple abelian variety with RM by $\cO_L$ over $\F_q.$ Then the supersingular part of its $p$-divisible group $A[p^{\infty}]_{\sups}$ lifts to characteristic zero with all of its endomorphisms. 
\end{prop}

\begin{proof}
We work over the algebraic closure $k$. By the previous lemma, we know that $A$ is isogenous to $B^m$ for some simple  abelian variety $B$ that is RM by $\mathcal{O}_{L^{\prime}}$ over $k$.  Let $\varphi: B^m\rightarrow A$ be an isogeny and let $\varphi[p^{\infty}]:  B[p^{\infty}]^m\rightarrow A[p^{\infty}]$ be the corresponding isogeny on the $p$-divisible groups. Note that since $p$ is split in $\cO_L$, it is completely split in $\cO_{L^{\prime}}.$ Hence, we write $B[p^{\infty}]_{\sups}^m = \prod_{1\leq i\leq a}\G^i_{\sups}\times \cdots \times \G^i_{\sups}$, the $m$-fold product of  $\G^i_{\sups}$, that is, the factor of the super-singular part of $B[p^{\infty}]$.
We know that $\Z_p^m$ acts on  $(\G^i_{\sups})^m$ for each $1\leq i\leq a.$ As in Proposition \ref{endcom}, we can choose a basis of idempotents $e_1, e_2, \dots ,e_m$ of $\Z_p^m$. Without loss of any generality consider $i=1$ and let $(\D^1)^m$ be the corresponding Dieudonn\'e module of the $p$-divisible group $(\G^1_{\sups})^m$.  Since $e_1,e_2, \dots, e_m$ are idempotent elements, we see that $(\D^1)^m$ is isomorphic to the direct sum  $e_1\cdot (\D^1)^m\oplus\cdots \oplus e_m\cdot (\D^1)^m$. This gives us an isomorphism of $(\G^1_{\sups})^m$ with $\mathscr{H}_1\times\cdots \times \mathscr{H}_m$ such that the action of $\Z_p^,$ respects the direct sum. Here $\mathscr{H}_j$ is the $p$-divisible groups corresponding to the Dieudonn\'e module $e_j\cdot(\D^1)$ for $1\leq j\leq m$.  Now, let $G$ be the kernel of $\varphi$ such that the order of $G$ is a power of $p.$  Since the action of $\Z_p^m$ respects the direct sum, therefore, $G\simeq G_1\times\cdots \times  G_m $ where $G_j\subset \mathscr{H}_j$ is of $p$-power order for each $1\leq j\leq m$. Hence,  $A[p^{\infty}]\simeq  \mathscr{H}_1/G_1\oplus \cdots \oplus \mathscr{H}_m/G_m.$ Since, each $\mathscr{H}_j/G_j$ is isogenous to $\mathscr{H}_j$ their endomorphism algebra is commutative, and therefore they lift to $W^{\prime}$, along with all their endomorphisms by Lemma \ref{AAlift}.

\end{proof}

\begin{thm}\label{lift}
The $p$-divisible group $A[p^{\infty}]$ of $A$ has at most $2^a$ lifts to $W^{\prime}$ where $g-a$ is the $p$-rank of the abelian variety. Furthermore, each lift $\widetilde{\G}$ of $A[p^{\infty}]$ is characterized by the property that every endomorphism of $A$ lifts.

\end{thm}
\begin{proof}
We know that for ordinary abelian varieties, there exists a  Serre-Tate ``canonical lift"  of their $p$-divisible group to $W(k)$. We denote the corresponding lift of $\G_{\text{\'et}}\times \G_{\tor}$ by $\widetilde{\G_{\text{\'et}}}\times \widetilde{\G_{\tor}}.$ Since each of the super-singular factors of $A[p^{\infty}]$ has at most $2$ lifts, we see that $\G_{\sups}$ has at most $2^a$ lifts. Moreover, the Serre-Tate canonical lift is characterized by the fact that every endomorphism also lifts. This is guaranteed by the previous Lemma \ref{AAlift} for the lift of the supersingular part as well. 
\end{proof}
\begin{defn}
We will refer to the lift(s) of $A[p^{\infty}]$ (and hence of A) constructed in Theorem \ref{lift} as \textit{canonical lift(s)}.

\end{defn}

\section{Classification}\label{3}
\noindent Analogous to the work of Deligne for ordinary abelian varieties in \cite{De69} and  Oswal-Shankar \cite{OS20} for  (simple) almost ordinary abelian varieties over finite fields, we provide a classification for abelian varieties with RM by $\cO_L$, not necessarily simple. However, for the sake of simplicity, we first consider $A$ to be a geometrically simple abelian variety with RM over $\F_q$. Let $\widetilde{A}$ be one of the canonical lifts of $A$. We fix an embedding $i: W^{\prime}\hookrightarrow \C$ so that we can consider $\widetilde{A}$ to be an abelian variety over $\C$ by base changing. Consider its integral homology $T(A) = H_1(\widetilde{A}\otimes_{i} \C, \Z)$ which is a free $\Z$ module of rank $2g.$ Since all the endomorphisms of $A$ lift, we have a Frobenius morphism $F_A$ of $T(A)$ that is the lift of the Frobenius of $A.$ The pair $(T(A), F_A, \lambda_A,\iota_A)$ satisfies the following properties. 
\begin{enumerate}
\item $F_A$ acts semi-simply on $T(A)\otimes \Q$.
\item There exists $V\in \End_\Z(T(A))$ such that $F\circ V = q = V\circ F.$
\item The characteristic polynomial $h(x)\in \Z[x]$  of $F_A$ is a Weil $q$-polynomial which is irreducible over $\Q.$ Furthermore, $h(x)$ factors as $h_{\ord}\prod_{1\leq i\leq a}h^i_{\sups}$ over $\Z_p$ where $1\leq i\leq a$. Each $h^i_{\sups}$ is a degree $2$ polynomial. 
    The polynomial $h(x)$ has $g-a$ roots that are $q$-adic units, $g-a$ roots   that have $q$-adic valuation $1$ and $2a$ roots with valuation $1/2$ in $\overline{\Q}_p.$ Here $g-a$ is the $p$-rank of the associated abelian variety.  Each of the factors $h^i_{\sups}$ is irreducible over $\Q_p$. Otherwise, the supersingular factor of $A$ has endomorphism by $\Q_p\times \Q_p$, but that is clearly not the case. Furthermore,  $h^i_{\sups}$ does not have $\pm{\sqrt{q}}$ as its roots, otherwise $A\times_{\F_q}\F_{q^2}$ won't be simple.
\item Since all the endomorphisms of $A$ lift, we get an embedding $\iota_A: \cO_L\hookrightarrow \End_{\Z}(T(A)).$
\item We note that $T(A)$ decomposes over $\Q_p$ as $T_{\ord} \oplus T^{1}_{\sups}\oplus\dots\oplus T^{a}_{\sups}$ where $T_{\ord} = \ker (h_{\ord}(F_A))$ and  $T^{i}_{\sups}=\ker (h_{\sups}(F_A))$ for each $1\leq i\leq a.$ Furthermore by the argument in $\ref{endcom}$, it follows that the endomorphism ring of $T_{1/2}$ is the maximal order $\cO_{\sups}$.
    
\item The map $\lambda_A$ identifies $\mathfrak{a}$ with a set of Riemann forms on $T(A)$. 

\end{enumerate}
We note that for an abelian variety $A$ with RM isogenous to a power of a simple abelian variety $B^m$, since $f_{A} = (f_B)^m$, its Deligne module $T(A)$ is isogenous to the product $\oplus_{i=1}^m T(B)$.

\begin{defn}\label{dmodrm}(Deligne Module with RM) Let $\mathfrak{a}$ be a fractional ideal of $L.$ 
\begin{enumerate}
    \item A pair $(T, F,\lambda,\iota)$ satisfying (1), (2), (3), (4), (5) and (6) is said to be a \textit{simple Deligne module with RM} by $\cO_L$. 
    
 \item An arbitrary Deligne module with RM is isogenous to a direct sum of simple Deligne modules. 
 \item A morphism of two Deligne modules with RM is given by a map $\varphi: (T,F, \lambda,\iota)\rightarrow (T^{\prime}, F^{\prime}, \lambda^{\prime},\iota^{\prime})$ such that 
 $\varphi$ is compatible by the action of $\cO_L$ on both the modules and $\varphi\circ F= F^{\prime}\circ \varphi.$ Analogous to abelian varieties, we have $\varphi^{*}\lambda^{\prime} = \alpha\lambda$ for some $\alpha\in L^{\times}$.
 \item An isogeny of Deligne modules with RM is defined to be a morphism  $\varphi: (T,F,\lambda,\iota)\rightarrow (T^{\prime}, F^{\prime}, \lambda^{\prime},\iota^{\prime})$ such that $\varphi\otimes \Q: T\otimes \Q\rightarrow T^{\prime}\otimes \Q$ is an isomorphism and $\varphi^{*}\lambda^{\prime} = \alpha\lambda$ for some $\alpha\in L^{\times}$.
\end{enumerate}
\end{defn}
\begin{defn}\label{catdef}(The category $\mathcal{L}_h$ and $\mathcal{C}_h$). 
\begin{enumerate}
\item Let $h(x)\in \Z[x]$ denote the polynomial in property $(3).$ The we define $\mathcal{L}_h$ to be the category of simple Deligne modules with RM with Frobenius polynomial $h$. Similarly, we define $\mathcal{C}_h$ to be the category of polarized abelian varieties with RM over $\F_q$ with Frobenius polynomial $h(x).$ Morphisms in $\mathcal{C}_h$ are also polarized, i.e., they preserve polarizations up to scaling by elements in $L^{\times}$.
\item We say $\mathcal{C}_h$ or $\mathcal{L}_h$ is totally ramified if $K_{\sups}$ (respectively inert) is a direct sum ramified (respectively) inert extensions of $\Q_p$. We will use the term ramified for either of the categories when $K_{\sups}$ has at least one ramified factor.  
\end{enumerate}
\end{defn}
\begin{prop}\label{surjfuct}
Every  Deligne module $(T,F,\lambda,\iota)$ with RM by $\cO_L$ in $\mathcal{L}_h$ arises from an abelian variety with RM by $\cO_L$.
\end{prop}

\begin{proof}
We first prove the result for a simple Deligne module, and the non-simple case follows by taking the product of the corresponding simple abelian varieties. So assume $(T,F)$ is a simple Deligne module.
Since $F$ satisfies a Weil-$q$ polynomial, using the theorem of Honda and Tate, we can find a simple abelian variety $A$ over $\F_q$ such that the characteristic polynomial of $\pi_A$ is the same as that of $F.$ We note that up to isogeny, we can always assume that $A$ has endomorphism by the maximal order in $K$; hence, it has RM by $\cO_L$. Therefore,  $(T(A)\otimes \Q, \pi_A) \simeq (T\otimes \Q, F).$  We wish to find an abelian variety $B$,that is $\cO_L$-isogenous to $A$ such that $(T(B),\pi_B)\simeq (T, F)$. It is sufficient to assume that $T(A)\subseteq T.$  Let $\widetilde{H}\subset \widetilde{A}$ be the subgroup corresponding to $T/T(A)$. Since $T$ and $T(A)$ are both $\cO_L$-modules, $\widetilde{H}$ is stable under the action of $\cO_L $ and $\pi_A.$ Let $H$ be the subgroup obtained by reducing $\widetilde{H}$ modulo $p.$ If $\widetilde{H}$ is of order prime-to-$p$, then $\widetilde{H}$ is a lift of $H$ as there exists a unique prime-to-$p$ subgroup in $\widetilde{A}[n]$ for some $(n,p)=1$ that will reduce to $H$ modulo $p.$  In the case when  $\widetilde{H}$ has order which is a power of $p,$ we write $H = H_{\text{\'et}}\times H_{\tor}\times H_{\sups}$ where $H_{\sups} = \prod_{i=1}^aH^{i}$. This decomposition follows because $H$ is stable under the action of $\cO_L.$ We note that the \'etale part and its dual lift uniquely to $\widetilde{H}$. For the supersingular components, 
say, $H^1_{\sups}$, we know that $T\otimes\Z_p$ and $T(A)\otimes\Z_p$ both have endomorphisms by the maximal order $\cO_{\sups}.$ Therefore, $T(A) = \overline{\omega}^r T$ where $\overline{\omega}$ is the uniformiser of $\cO_{\sups}$ and $r$ is a positive integer.  By construction, the subgroup $H$ corresponds to the kernel of $\overline{\omega}^r.$ Since this endomorphism lifts uniquely to $\widetilde{A}[p^{\infty}]$, we see that $\widetilde{H}^1_{\sups}$ must be the kernel of the lift of $\overline{\omega}^r$. Finally, if we consider $\widetilde{A}/\widetilde{H} = \widetilde{B},$ then its reduction corresponds to $A/H =B$ such that $(T(B), \pi_B) = (T,F).$
By construction the polarisation data on Deligne modules is the same as the data of polarisation on RM-abelian varieties. That is, the data on both are parametrized by the fixed fractional ideal $\mathfrak{a}$.  This completes the proof of proposition \ref{surjfuct}.
\end{proof}

\subsection{Totally ramified classes}\label{totram} Let $\mathcal{C}_h$ denote a totally ramified isogeny class. As mentioned before, because any abelian variety in $\mathcal{C}_h$ has $2^a$ lifts to $W(k),$ the functor from $\mathcal{C}_h$ to $\mathcal{L}_h$ depends on the choice of the lift of $A.$ However, we will show that once we choose a lift $\widetilde{A}$ of $A$, then we can canonically determine a lift of every member in the isogeny class. This is because we can lift every subgroup canonically to $\widetilde{A}$ as we prove below in proposition \ref{liftsubgrp}. Therefore, up to a choice of the lift, the association from $\mathcal{C}_h$ to $\mathcal{L}_h$ is functorial.

\begin{prop}\label{liftsubgrp}
Let $\mathcal{C}_h$ be a totally ramified isogeny class. Suppose $A\in \mathcal{C}_h$ and $\widetilde{A}$ be one of its lifts to $W(k).$ Let $G \subset A$ be the kernel of an isogeny in $\mathcal{C}_h$. Then there exists a canonical subgroup $\widetilde{G} \subset \widetilde{A}$ lifting $G$. 
	
\end{prop}

\begin{proof}
 	Since the isogenies in $\mathcal{H}_{L,\mathfrak{a}}$ are compatible with the action of $\cO_L,$ we see that $G$ remains fixed by its action and therefore,  it splits as $G_{\operatorname{et}} \times G_{\operatorname{tor}} \times \prod_{i=1}^{a}G_{\sups}^{i}$. It suffices to prove this result for each component. The result is certainly true for prime-to-$p$ subgroups of $A$. Therefore we assume that the order of $G$ is a power of $p$. We check this on the level of $p$-divisible groups. We know that the canonical lift of $A[p^{\infty}]=\G_{\operatorname{\text{\'et}}} \times \G_{\operatorname{tor}} \times \prod_{i=1}^{a}\G_{\operatorname{\sups}} $ to be the product of the canonical lifts of its \'etale, multiplicative and the local-local part. The \'etale and multiplicative subgroups lift uniquely by the definition of the canonical lift, so we are left to check that every local-local finite flat subgroup of $A$ lifts uniquely to $\widetilde{A}$.   Since $G_{\operatorname{\sups}} \subset \G_{\operatorname{\sups}}$, it suffices to check that $\G_{\operatorname{\sups}}$ lifts to a subgroup $\widetilde{G}_{\operatorname{\sups}} \subset \widetilde{\G}_{\operatorname{\sups}} $. For each $1\leq i\leq a$, the $p$-divisible group $\G_{\sups}^{i}$ is a connected and $1$-dimensional group. Therefore it has a unique order $p^{r}$ subgroup for each positive integer $r$. Since each $\G_{\sups}$ is ramified, we see that this has to be the kernel of the endomorphism $\overline{\omega}^{r}$ where $\overline{\omega}$ is the uniformizer of $\cO_{\sups}^{i}$. The canonical lift $\widetilde{G}_{\sups}^{i}$ has the property that every endomorphism of $\G_{\sups}^{i}$ lifts and hence it follows that $\widetilde{G}=\widetilde{G}_{\sups}^{i}[\overline{\omega}^{r}]$ is the required lift of $G$.
\end{proof}
\begin{defn}(Lifts of objects in the totally ramified isogeny class) 
 	\begin{enumerate}
 		\item Let $G \subset A$ be a finite flat subgroup as in Proposition 4.3. We define $\widetilde{G} \subset \widetilde{A}$ to be the canonical lift as in the same proposition.
 		\item Let $B \in \mathcal{C}_{h}$ be an abelian variety over $\F_q$ such that there exists an isogeny in $\mathcal{C}_h$ $\varphi : A \to B$ with kernel $G$, i.e. it respects the $\mathcal{O}_L$-actions. We define $\widetilde{B}$ to be the lift of $B$ given by $\widetilde{A}/\widetilde{G}$.
 	\end{enumerate}
 \end{defn}

\begin{prop}
The lift $\widetilde{B}$ is the canonical lift of $B$ such that all the endomorphisms also of $B$ also lift to $\widetilde{B}$. Furthermore, $\widetilde{B}$ does not depend on the choice of the isogeny $\varphi$.	
\end{prop}
 \begin{proof}
 	We check this on the level of $p$-divisible groups. We have $\widetilde{B}[p^{\infty}]=\widetilde{\G}_{B,\operatorname{et}} \times \widetilde{\G}_{B, \operatorname{tor}} \times \prod_{i=1}^{2^{a}}\widetilde{\G}^{i}_{B, \operatorname{\sups}}$ where $\tilde{\G}_{B, \heartsuit}$ is the canonical lift of $\G_{B, \heartsuit}$ for $\heartsuit$  either $\et, \tor$ or $\sups$. By Proposition $\ref{liftsubgrp}$, the subgroup  $\widetilde{G}$ is the canonical lift of $G$ which is preserved under the action of $\cO_{\sups}$, which implies that it acts on $\widetilde{B}[p^{\infty}]$ as well. Therefore $\widetilde{B}$ is a canonical lift of $B$.
 	
 \noindent For the second claim, we take $\varphi_{1}$ and $\varphi_{2}$ to be two isogenies from $A$ to $B$ with kernels $G_{1}$ and $G_{2}$ respectively. Let $\widetilde{B}_{i}=\widetilde{A}/\widetilde{G}_{i}$ for $i=1,2$ be the corresponding abelian varieties isogenous to $\widetilde{A}.$ Without loss of generality we may assume that $\varphi_{2}=\alpha \circ \varphi_{1}$ where $\alpha \in \End(B)$, by replacing $\varphi_{2}$ with an integer scalar multiple. Since every endomorphism in $\End(B)$ lifts to $\End(\widetilde{B})$ as a canonical lift, the result follows.  
 \end{proof}

\begin{prop}
Let $B$ and $C$ be abelian varieties in the totally ramified isogeny class $\mathcal{C}_{h}$ of $A$. Then every $\F_q$-isogeny $\varphi: B \to C$ lifts uniquely to an isogeny $\widetilde{\varphi}: \widetilde{B} \to \widetilde{C}$.  

\end{prop}

\begin{proof}
Let $\psi : A \to B$ be an $\F_q$-isogeny. Denote by $G$ and $H$, the kernel of $\ker(\psi)$ and  $\ker(\varphi \circ \psi)$ respectively. By Proposition \ref{liftsubgrp}, there are unique lifts of $B$ and $C$ respectively, namely $\widetilde{B}=\widetilde{A}/\widetilde{G}$ and $C= \widetilde{A}/\widetilde{H}$. Therefore we take $\widetilde{\varphi}: \widetilde{A}/\widetilde{G} \to \widetilde{A}/\widetilde{H}.$
	\end{proof}

\noindent Hence for an $\F_q$ isogeny $\varphi: B \to C$, we have a map on the integral homology 
$$\mathcal{F}(\varphi): H_{1}(\widetilde{B} \otimes_{\epsilon}\C, \Z) \to H_{1}(\widetilde{C} \otimes_{\epsilon}\C, \Z)$$
which is compatible with the Frobenius action. This gives a functor on the two categories in definition \ref{catdef},
$$\mathcal{F}: \mathcal{C}_{h} \to \mathcal{L}_h$$
$$A\mapsto H_{1}(\widetilde{A} \otimes_{\epsilon}\C, \Z)$$

\begin{prop}
Let $A$ be an abelian variety in the totally ramified class $\mathcal{C}_h$. Fix a lift $\widetilde{A}$ of $A.$ Then the functor $\mathcal{F}$ mentioned above is an equivalence of categories.
\end{prop}

\begin{proof}
By Proposition \ref{surjfuct}, the functor $\mathcal{F}$ is essentially surjective. The proof of  \cite[Proposition 3.4]{OS20} shows that $\mathcal{F}$ is fully faithful as well.
\end{proof}
 
\subsection{Inert isogeny classes}\label{inert}
 We next consider the case where $\mathcal{C}_h$ is inert. That is, the $\Z_p$-algebra $\cO_{\sups}^i$ is inert for all  $1\leq i\leq a$. Let $A,B \in \operatorname{Ob}(\mathcal{C}_{h})$ and $\varphi : A \to B$ be an $\F_q$ isogeny with kernel $G$. Let $\varphi[p^{\infty}]: A[p^{\infty}] \to B[p^{\infty}]$ be the induced map of $p$-divisible groups. We know that, the map $\varphi$ commutes with actions by $\cO_{L}$ and therefore $G$  splits as a product $G_{\ord}\times \prod_{1\leq i\leq a}{G^i}_{\sups}$. As a consequence,  $\varphi[p^{\infty}]$ and its kernel also decompose accordingly. 

\begin{defn}
We define an equivalence relation $\sim$ on the set of objects of $\operatorname{Ob}(\mathcal{C}_{h})$ as follows. For $A,B \in \operatorname{Ob}(\mathcal{C}_{h})$ we say $A \sim B$ when some (hence \textit{every}) $\F_{q}$-isogeny $f: A \to B$ is such that if $f[p^{\infty}]: A[p^{\infty}] \to B[p^{\infty}]$ has all its supersingular components $f_{\sups}^{j}[p^{\infty}]: A_{\sups}^{j}[p^{\infty}] \to B_{\sups}^{j}[p^{\infty}]$ such that the kernel $\ker(f_{\sups}^{j}[p^{\infty}])$ has order equal to an even power of $p.$
\end{defn}
\noindent The equivalence relation partitions $\operatorname{Ob}(\mathcal{C}_{h})$ into $2^{a}$ equivalence classes and we denote the corresponding full subcategories by $\mathcal{C}_{1,h},\cdots, \mathcal{C}_{2^{a},h}$. Therefore $\operatorname{Ob}(\mathcal{C}_{h})=\bigcup_{i=1}^{2^{a}}\operatorname{Ob}(\mathcal{C}_{i,h})$. With this notation, we make the following claim. 

\begin{prop}
	By restricting $\mathcal{F}$ to each equivalence class $\mathcal{C}_{i,h}$ for $1\leq i\leq 2^a$, the association 
 $$\mathcal{F}_i: A \mapsto (T(A), F(A))$$ 
is functorial. Moreover, it induces an equivalence of categories.
 
\end{prop}
\noindent The proof of this proposition follows verbatim as in \cite[Proposition 3.4]{OS20}.


\subsection{Ramified isogeny classes}
Finally, the results in subsections \ref{totram} and \ref{inert} can be used to similarly prove that for a ramified isogeny class $\mathcal{C}_h$, the association $A\mapsto (T(A), F(A))$ gives an equivalence of categories once we fix a lift of an abelian variety $A\in \operatorname{Ob}(\mathcal{C}_h)$ \textit{and} define an equivalence relation on the objects as in the inert case. More precisely, suppose we have $k$ ramified components in $K_{\sups}$ out of the $a$ supersingular components.
Let $\tilde{A}$ denote one of the $2^k$ canonical lifts of $A.$ We call $\tilde{A}$ \textit{the} canonical lift of $A$. Furthermore,  we partition $\mathcal{C}_{h}$ into $2^{a-k}$ equivalence classes using the relation $\sim$ defined in \ref{inert}. Analogous to the inert case, we write $\operatorname{Ob}(\mathcal{C}_h)=\bigcup_{j=1}^{2^{a-k}}\operatorname{Ob}(\mathcal{C}_{j,h})$. 

\vspace{2mm}
\noindent Following the same proofs as in subsection \ref{totram}, we can show that for each $1\leq j \leq a-k$, fixing a lift of an object $A_j\in \operatorname{Ob}(\mathcal{C}_{j,h})$, fixes the lift of every element in the isogeny class every isogeny between the objects in $\mathcal{C}_{j,h}.$ Now, analogous to the inert case, we restrict $\mathcal{F}$ to each subcategories $\mathcal{C}_{j,h}$. This gives a functor 
$$\mathcal{F}_j: \mathcal{C}_{j,h} \to \mathcal{L}_h$$ 
for all $j$ such that $1 \le j \le 2^{k-a}.$ Finally, as is proved before, we have the following proposition in the ramified case as well.
\begin{prop}
The functor $\mathcal{F}_j: \mathcal{C}_{j,h} \to \mathcal{L}_h$ is an equivalence of categories for all $1 \le j \le 2^{k-a}.$	

\end{prop}


\section{Size of isogeny classes}

%

\noindent The goal of this section is to prove Theorem $\ref{isogsize}$. Throughout the section, we fix a geometrically simple abelian variety $A$ over $\F_q$ and define $R_{n}$ as the smallest order in $\Q(\pi_A)$ containing $\pi_A^{n}$, $q^{n}/\pi_A^{n}$ and $\cO_L$ such that $R_{n} \otimes \Z_{p}$ contains $\cO_{\sups}$. By Theorem \ref{mainthm}, the set of $\F_{q^{n}}$-points in $\mathcal{H}_{L,\mathfrak{a}}$ isogenous to $A$ with the endomorphism ring exactly $R_{n}$ is in bijection with the set of isomorphism classes of finitely generated $R_{n}$-submodules of $\Q(\pi_A)$. 
In fact, the following lemma showsthat $R_n$ is a Gorenstein ring. Hence, every such $R_n$-sub-module is invertible. 

\begin{lem}\label{lemmagor}
	Let $K$ be a CM field and $L$ be its totally real field. Let $D$ be a Dedekind domain with the field of fractions $L$. For any $D$-subalgebra $R$ inside the integral closure of $D$ in $K$ such that $R \otimes \Q=K$, $R$ is a Gorenstein ring.
\end{lem}
\begin{proof}
We know that a ring $R$ is Gorenstein if and only if all its localizations are Gorenstein. Thus, without loss of generality, we assume that $D$ is a principal ideal domain. Since $R$ is a free $D$-module of rank $2$, there exists a basis $\{1,\alpha\}$ such that $R=D[\alpha]$. Let $f \in D[X]$ be the minimal polynomial of $\alpha$ such that $R=D[X]/fD[X]$. Write $\alpha=(X \text{ mod }f) \in R$ then the complementary module $R^{\dagger}=f^{\prime}(\alpha)^{-1}R$ is invertible. Therefore, \cite[Proposition 2.7]{BL94} implies that $R$ is a Gorenstein ring. 
\end{proof}

\begin{cor}\label{gor}
$R_{n}$	is a Gorenstein order inside $\cO_{K}$.
\end{cor}

\begin{proof}
This comes directly from Lemma \ref{lemmagor}. 
One can also get proof by checking that $R \otimes \Z_{\ell}$ is monogenic for every $\ell \ne p$ in the context of \cite{KK18} and applying the fact that monogenic orders are Gorenstein.
\end{proof}
\noindent Let $I\subset K$ be an $R_n$-submodule. By Corollary \ref{gor}, we see that $I$ is invertible if there exists some fractional ideal $J$ with $IJ = R_n$. We say two fractional ideals, $I_1$ and $I_2$ are equivalent if and only if there exists some $a\in K^{\times}$ such that $I_1 = aI_2.$ We define the equivalence classes of all such fractional ideals of $R_n$ to be the \textit{class group of $R_n$}. We use the standard notation $\mathcal{C}\ell(R_n)$ to denote the class group. 
\subsection{The lower bound}

We now have a bijection between $\F_{q^{n}}$-isomorphism classes of abelian varieties in $\mathcal{C}_h$ having endomorphism ring exactly $R_{n}$ and the ideal class group $\mathcal{C}\ell(R_{n})$. The class group $\mathcal{C}\ell(R_n)$ is well approximated by the square root of its discriminant as $n$ approaches infinity. Therefore, it suffices to compute the (square root) of the discriminant of $\cO_{L}[\pi_A^{n}]$ and the index of $\cO_{L}[\pi_A^{n}]$ inside $R_{n}.$
The discriminant of $R_n$ is then approximated by dividing the two estimates. We recall that $\pi_A$ is the Weil $q$-integer of $A$. Let $\alpha_{1}, \alpha_{2}, \cdots, \alpha_{g}, \overline{\alpha_{1}}, \cdots, \overline{\alpha_{g}}$ be the image of $\pi_A$ under the $2g$ complex embeddings of $L$. Let $\theta_{1}, \cdots, \theta_{g}$ denote the arguments of $\alpha_{1},\cdots,\alpha_{g}$ respectively. Note that for each $i$, $\overline{\alpha_i}=q/\alpha_i$ is the complex conjugate of $\alpha_{i}$. We have the following lemma which points out the fact that for all $1 \le i \le g$ and for all but finitely many $n$, $n\theta_{i}$ (hence $\sin n\theta_{i}$) cannot be too small. 

\begin{lem}\label{linear approximation}
For every $\epsilon > 0,$ and $1 \le i \le g$, let $\omega_{n,i} \in [0,2\pi)$ be the unique real number such that $n\theta_{i} \equiv \omega_{n,i} \pmod{2 \pi} $. We have $|\omega_{n,i}| > \frac{1}{q^{n\epsilon}}$ for all but finitely many $n$.
\end{lem}
\begin{proof}
For each $1\leq i\leq g$, the argument $\theta_i$ can be viewed as the logarithm of $\beta_i:=\alpha_i/\overline{\alpha_i}$, which is same as $\log(\alpha_i) - \log(\overline\alpha_i)$.
We first prove that $\log(\alpha_i)$ and $\log(\overline\alpha_i)$ are linearly independent over $\mathbb{Q}$. Suppose we have $a,b \ne 0$ such that $$a\log(\alpha_i) +b \log (\overline\alpha_i) = 0.$$ This is,  $$\alpha_i^a = \overline\alpha_i^{-b}.$$ Comparing the $q$-power we see that $a = -b$. This only happens when $b\theta_i=k\pi$ for some integer $k$. Indeed, $\theta_i/\pi \in \Q$ if and only if $\beta_i =\alpha_i/\overline{\alpha_i}$ is a root of unity for each $i$ if and only if the Frobenius angle rank of $A$ equals zero. This happens and only if $A$ is supersingular. Since we have assumed that $A$ is geometrically simple, this is true only when $g=1$ and $A$ is a supersingular elliptic curve. See section 3 of \cite{DKZ} for details.

\noindent Using Baker's theorem for the absolute value of linear combinations of logarithms of algebraic numbers, we see that there exists some absolute constant $c$ such that for every $\epsilon>0$,
$$|\omega_{n,i}|  = |n\log(\alpha_i) - n(\log q/\alpha_i)|>\frac{1}{n^{c}}>\frac{1}{q^{n\epsilon}}$$ for all but finitely many $n$.
\end{proof}

\begin{lem}\label{4.3}
 We have $\operatorname{disc}(\cO_{L}[\pi_A^{n}])=q^{n(g+o(1))}$ as $n \to \infty.$
	\end{lem}
	\begin{proof}
Consider the tower of extensions of rings $\Z \subset \cO_{L} \subset \cO_{L}[\pi_A^{n}]$. We know that $$\operatorname{disc}(\cO_{L}[\pi_A^{n}])=\operatorname{disc}(\cO_{L})^{2}N_{\cO_{L}/\Z}(\operatorname{disc}_{\cO_{L}}(\cO_{L}[\pi_A^{n}]))$$
See, for example, \cite[Prop. III.8]{Se79}.
Consider a basis  $\{1,\pi_A^{n}\}$ of $\cO_{L}[\pi_A^{n}]$ over $\cO_L$.
We compute,

\begin{align*}
	N_{\cO_{L}/\Z}(\operatorname{Disc}_{\cO_{L}}(\cO_{L}[\pi_A^{n}])) &= N_{\cO_{L}/\Z}((\pi_A^{n}-q^{n}/\pi_A^{n})^{2}) \\
	&=\prod_{j=1}^{g}(\alpha_{j}^{n}-q^{n}/\alpha_{j}^{n})^{2}\\
	&= \prod_{j=1}^{g}4q^{n}(i\sin n\theta_{j})^{2}
  \end{align*}

\noindent This along with the Lemma \ref{linear approximation}, completes the proof.

	\end{proof}
	
	\begin{lem}\label{lowerboundclassgroup}
		For all but finitely many $n$, we have $\mathcal{C}\ell(R_{n})=q^{\frac{n}{2}(g-a+o(1))}$.
	\end{lem}
	
\begin{proof}
	In order to estimate the index of $\cO_{L}[\alpha^n]$ inside $R_{n}$, we compute the corresponding index locally at $p$. Without loss of any generality, we do the computation for $n=1$. Let $p(x)$ denote the minimal polynomial of $\alpha$ over $\cO_{L}$ which we know is degree $2.$
	Let $\sigma_{1},\cdots,\sigma_{g}$  denote the restrictions of $g$ embeddings  $L\hookrightarrow \overline\Q_p$ to $\cO_L.$ Then we have 
	
	\begin{align*}
		\cO_{L}[\alpha^{n}] \otimes \Z_{p} &= \frac{\cO_{L}[x]}{h(x)} \otimes \Z_{p} \\
		&= \frac{\Z_{p}[x]}{\sigma_{1}(p(x))} \oplus \cdots \oplus \frac{\Z_{p}[x]}{\sigma_{g}(p(x))}
 	\end{align*}
since  $p$ splits completely in $\cO_L$.	
Let $f(x)$ be the minimal polynomial of $\alpha$ over $\Z_{p}$. We know that $f(x)=\sigma_{1}(p(x))\cdots \sigma_{g}(p(x))$ gives the decomposition of $p(x)$ over $\cO_L\otimes \Z_p.$  Without loss of any generality, we take $\sigma_{1}(p(x)), \cdots, \sigma_{a}(p(x))$ to be the $a$ supersingular factors of $p(x)$. Let $\delta_{1,1}, \delta_{1,2},\cdots,\delta_{a,1},\delta_{a,2}$ to be the corresponding roots of the polynomials. Let $g_{i}(x)$ denote the polynomial with roots $\delta_{i,1}/q^{\frac{1}{2}}$ and $\delta_{i,2} /q^{\frac{1}{2}}$.  The index of $\cO_{F}[\alpha] \otimes \Z_{p}$ inside $R_{n} \otimes \Z_{p}$ is approximated by  dividing the product of the discriminants of $\sigma_i(p(x)$ and $g_i(x)$
which gives the following expression.

$$\prod_{i=1}^{a}\frac{{(\delta_{i,1}-\delta_{i,2})}}{(\delta_{i,1}/q^{\frac{1}{2}}-\delta_{i,2}/q^{\frac{1}{2}})}=q^{\frac{a}{2}}$$
The lemma now follows by putting together the above calculation with Lemma \ref{4.3}. 
\end{proof}
\subsection{The upper bound}\label{upperbound}
In this section, we show that the lower bound computed in the above section is indeed sharp. We will do this by approximating the class groups of all over orders  of $R_n$ in order to count the $\F_{q^n}$-isomorphism classes of abelian varieties in $
\mathcal{C}_h$ with endomorphism ring $\cO$ where $\cO$ is any arbitrary over-order of $R_n$.  We recall some standard notation first.

\begin{defn}
	Let $R$ be an order in a number field $K$. Let $\mathcal{I}(R)$ be the set of all fractional ideals of $R$ The ideal class monoid of $R$, denoted by ICM(R) is defined to be 
	
	$$\operatorname{ICM}(R)=\mathcal{I}(R)/\simeq$$
where $I \simeq J$ if and only if there exists an $\alpha \in K^{*}$ such that $I=\alpha J$.
	
\end{defn}

\begin{defn}
For all \textit{invertible} factional ideals, we analogously define the Picard group of $R$ to be
$$\mathcal{C}\ell(R)=\{\text{invertible fractional ideal } I \in \mathcal{I}(R) \} / \mathcal{P}(R)$$
where $\mathcal{P}(R)$ is the group of of principal fractional $R$-ideals. 	
\end{defn}

If two fractional $R$-ideals are isomorphic, then they have the same multiplicator ring. It follows that 
$$\bigsqcup_{R \subset \cO \subset \cO_{K} }\mathcal{C}\ell(\mathcal{O}) \subseteq \operatorname{ICM}(R)$$
where the disjoint union is taken over all over-orders $\cO$ of $R$.

\begin{defn}
An order $R$ is called a \textit{Bass order} if every over-order of $R$ is Gorenstein, or equivalently, if $\cO_{K}/R$ is cyclic, which is to say $\cO_{K}=R+xR$ for some $x \in \cO_{K}$.
\end{defn} 

We have the following lemma, which gives another useful characterization for Bass orders.
\begin{lem}
The following statements are equivalent:
\begin{itemize}\label{bass}
	\item[a.] $R$ is a Bass order,
	\item[b.] $\bigsqcup_{R \subset \cO \subset \cO_{K} }\mathcal{C}\ell(\mathcal{O}) = \operatorname{ICM}(R).$
\end{itemize}	
\end{lem}

For proof of this proposition, see \cite[Proposition 3.7]{Mar20}. Finally, we have the following proposition that gives us an upper bound on the size of the isogeny classes. 

\begin{prop}\label{upbound}
	For any positive integer $n$, we have $$N(A, q^{n}) \le (q^{\frac{n}{2}})^{(g-a+o(1))}$$
\end{prop}

\begin{proof}
From Lemma \ref{lemmagor} we deduce that every order between $\cO_{F}[\alpha^{n}]$ and $\cO_{K}$ is Gorenstein, hence every element of $\operatorname{ICM}(R_{n})$ lies in $\mathcal{C}\ell(\cO)$ for some  unique order $\cO$ over $R_{n}$. Therefore by Lemma \ref{bass} we have 	
$$\bigsqcup_{R_{n} \subset \cO \subset \cO_{K} }\mathcal{C}\ell(\mathcal{O}) = \operatorname{ICM}(R_{n}).$$

Let $h(\cO)$ be the class number of $\cO$. We note that the class group of the smallest order $R_{n}$ has the largest size. By the lattice isomorphism theorem, subgroups of $\cO_K$ containing $R_{n}$ are in bijection with subgroups of $\cO_{K}/R_{n}$. Since $R_n$ is Bass, we have that $\cO_{K}/R_{n}$ is cyclic. Therefore its subgroups are indexed by $d\mid i_n$ where $i_n =|\cO_K/R_n|.$ Putting everything together, we see that 
$$N(A, q^{n})=\sum_{d \mid i_{n}}h(\cO_{d})\le i_{n}^{o(1)}h(R_{n})=(q^{\frac{n}{2}})^{(g-a+o(1))}$$
where $\cO_{d}$ is the order corresponding to the subgroup of order $d$ in $\cO_{K}/R_n$.
\end{proof}
\subsection{The proof of Theorem \ref{isogsize}}

Let $R_{n}^{+}=R_{n} \cap L$ be the totally real part of $R_{n}$. Since $R_n$ contains $\cO_L,$ we have $R_n^+=\cO_L.$
\begin{prop}\label{norm map}
    The subset $I(A, q^{n})$ with endomorphism ring equal to $R_{n}$ is either empty, or admits a bijective map if $\mathcal{C}_{h}$ is totally ramified, or admits a $2^{k}$-to-one map if there are $k$ inert factors,  onto the kernel of the norm map
    $$N: \mathcal{C}\ell(R_{n}) \to \mathcal{C}\ell^{+}(R_{n}^{+}),$$
    where $\mathcal{C}\ell^{+}(R_{n}^{+})$ is the narrow class group of the totally real order $R_{n}^{+}$.
\end{prop}
\begin{proof}
    The proof follows verbatim as the proof of \cite[Proposition 3.5]{ST18}.
\end{proof}

One last thing to finish the proof of Theorem \ref{isogsize} is to prove that there exists some principally polarized abelian varieties with endomorphism ring equal to $R_{n}$.

\vspace{1em}
Let $\mathfrak{a}$, $\mathfrak{b}$ be fractional ideals of $\mathcal{O}_{L}$ and $z \in \mathcal{H}^{+}$ . For a lattice $\Lambda_{z}$ given by 
$$\Lambda_{z}=\mathfrak{a} \cdot z \oplus \mathfrak{b} \cdot 1,$$
recall that a polarization on $A_{z}=\mathbb{C}^{g}/\Lambda_{z}$ is given by a Riemann form on $\mathfrak{a} \oplus \mathfrak{b}$:
$$E_{r,z}((x_{1},y_{1}),(x_{2},y_{2})):= \operatorname{Tr}_{L/\Q}(r(x_{1}y_{2}-x_{2}y_{1}))$$
for some $r \in (\mathcal{D}_{L/\Q}\mathfrak{a}\mathfrak{b})^{-1}$. Therefore we have 
$$\Hom_{\mathcal{O}_{L}}(A_{z}, A_{z}^{\vee})^{symm}=(\mathcal{D}_{L/\Q}\mathfrak{a}\mathfrak{b})^{-1}$$ and the set of polarization is identified with elements in $(\mathcal{D}_{L/\Q}\mathfrak{a}\mathfrak{b})^{-1,+}$. Moreover, if this holds, then the degree of the polarization on $A_{z}$ is given by $E_{r,z}$ is Norm$(r\mathcal{D}_{L/\Q}\mathfrak{a}\mathfrak{b})$. In particular, there exists an $r$ such that $E_{r,z}$ is principal if and only if $\mathfrak{a}\mathfrak{b}=\mathcal{D}_{L/\Q}^{-1}$ in the narrow class group $\mathcal{C}\ell(L)^{+}$. In other words, we may take $\mathfrak{a} \oplus \mathfrak{b}=\mathcal{D}^{-1}_{L/\Q} \oplus \mathcal{O}_{L}.$ \cite[Chapter 2; Corollary 2.10]{Goren}.
\begin{prop}\label{polarizationRn}
    For every positive integer $n$, there exists some principally polarized abelian varieties with endomorphism ring equal to $R_{n}$.
\end{prop}

\begin{proof}
  Let $\mathfrak{i}$ be a fractional ideal of $\mathcal{O}_{L}$. A polarization on the abelian variety corresponding to $R_{n}$ can be realized as a $\Z$-valued bilinear form on $R_{n}$, which is of the form:
$$ \langle x,y \rangle \mapsto \operatorname{Trace}_{K/\Q}(ai x \bar{y}),$$
where,
\begin{itemize}
    \item[(1)] $\bar{y}$ is the complex conjugation of $y$.
     \item[(2)] $ \langle -,- \rangle$ restricted to $\mathfrak{a}$ is integral.
     \item[(3)] $a$ is totally positive element in $\mathfrak{i}$.
     \end{itemize}
In order to construct a principally polarized abelian variety, the argument above requires that  $\mathfrak{i}=\mathcal{D}_{L/\Q}^{-1}.$
Let $R_{n}^{\vee}$ be the dual of $R_{n}$ under the bilinear form induced by the trace map $(x,y) \mapsto \operatorname{Trace}_{K/\Q}(x,y).$ Let $R_{n}^{+}=R_{n} \cap L$ be the totally real part of $R_{n}$. We know that $R_n^+=\cO_L.$ Consider $R_{n}$ as a imaginary extension of degree $2$ over $\cO_L$. Therefore $R_{n}^{\vee}=2bi(\cO_L)^{\vee}$ for some positive integer $b$. Now, $(\cO_L)^{\vee} =  \mathcal{D}_{L/\Q}^{-1}$. Consider the bilinear form on $R_{n}$ given by 
$$(x,y) \mapsto \operatorname{Trace}_{K/\Q}(\lambda_{n}x\bar{y}),$$
where $\lambda_{n} \in R_{n}^{\vee}.$ Therefore, we can get a totally positive $\lambda_{n}$ once we choose an element $c$ from $\mathcal{D}_{L/\Q}^{-1}$ such that $\varphi(2bci)/i$ is totally positive for all $\varphi \in \Phi$. This gives a principal polarization on the abelian variety with endomorphism ring equal to $R_{n}$. 
\end{proof}

\begin{proof}[Proof of Theorem \ref{isogsize}]
    The theorem now follows from Lemma \ref{lowerboundclassgroup}, Proposition \ref{upbound}, Proposition \ref{norm map}, Proposition \ref{polarizationRn}, and the fact that the size of $ \mathcal{C}\ell^{+}(R_{n}^{+})$ does not grow with $n$ in our case because $R_{n}^{+} = \cO_L.$
\end{proof}

\bibliography{reference}

\begin{bibdiv}
\begin{biblist}

\bib{AC02}{article}{
      author={Achter, Jeffrey~D.},
      author={Cunningham, Clifton~L.R.},
       title={Isogeny classes of hilbert–blumenthal abelian varieties over
  finite fields},
        date={2002},
        ISSN={0022-314X},
     journal={Journal of Number Theory},
      volume={92},
      number={2},
       pages={272\ndash 303},
  url={https://www.sciencedirect.com/science/article/pii/S0022314X01927167},
}

\bib{AG04}{book}{
      author={Andreatta, Fabrizio},
      author={Goren, Eyal},
       title={Hilbert modular varieties of low dimension},
        date={2004},
        ISBN={978-3-11-017478-6},
}

\bib{BL94}{article}{
      author={Buchmann, J.~A.},
      author={Lenstra, H.~W.},
       title={Approximatting rings of integers in number fields},
        date={1994},
        ISSN={12467405, 21188572},
     journal={Journal de Théorie des Nombres de Bordeaux},
      volume={6},
      number={2},
       pages={221\ndash 260},
         url={http://www.jstor.org/stable/26273948},
}

\bib{De69}{article}{
      author={Deligne, Pierre},
       title={Variétés abéliennes ordinaires sur un corps fini},
        date={1969},
        ISSN={1432-1297},
     journal={Inventiones mathematicae},
      volume={8},
      number={3},
       pages={238\ndash 243},
         url={https://doi.org/10.1007/BF01406076},
}

\bib{DKZ}{article}{
      author={Dupuy, Taylor},
      author={Kedlaya, Kiran~S.},
      author={Zureick-Brown, David},
       title={Angle ranks of abelian varieties},
        date={2023},
     journal={Mathematische Annalen},
       pages={1432\ndash 1807},
         url={https://doi.org/10.1007/s00208-023-02633-7},
}

\bib{Goren}{book}{
      author={Goren, Eyal Z. (Eyal~Zvi)},
       title={Lectures on hilbert modular varieties and modular forms},
    language={eng},
      series={CRM monograph series, v. 14},
   publisher={American Mathematical Society},
     address={Providence, R.I},
        date={2002},
        ISBN={082181995X},
}

\bib{howe}{article}{
      author={Howe, Everett~W.},
       title={Principally polarized ordinary abelian varieties over finite
  fields},
        date={1995},
        ISSN={0002-9947,1088-6850},
     journal={Trans. Amer. Math. Soc.},
      volume={347},
      number={7},
       pages={2361\ndash 2401},
         url={https://doi.org/10.2307/2154828},
      review={\MR{1297531}},
}

\bib{KK18}{article}{
      author={Kleiman, S.L.},
      author={Kleppe, J.O.},
       title={Macaulay duality over any base},
        date={2018},
     journal={Preliminary preprint dated 26th December 2018},
}

\bib{Mar20}{article}{
      author={Marseglia, Stefano},
       title={Computing the ideal class monoid of an order},
        date={2020},
     journal={Journal of the London Mathematical Society},
      volume={101},
      number={3},
       pages={984\ndash 1007},
  eprint={https://londmathsoc.onlinelibrary.wiley.com/doi/pdf/10.1112/jlms.12294},
  url={https://londmathsoc.onlinelibrary.wiley.com/doi/abs/10.1112/jlms.12294},
}

\bib{OS20}{article}{
      author={Oswal, Abhishek},
      author={Shankar, Ananth~N.},
       title={Almost ordinary abelian varieties over finite fields},
        date={2020},
     journal={Journal of the London Mathematical Society},
      volume={101},
      number={3},
       pages={923\ndash 937},
  eprint={https://londmathsoc.onlinelibrary.wiley.com/doi/pdf/10.1112/jlms.12291},
  url={https://londmathsoc.onlinelibrary.wiley.com/doi/abs/10.1112/jlms.12291},
}

\bib{Se79}{book}{
      author={Serre, Jean-Pierre},
       title={Local fields},
      series={Graduate Texts in Mathematics},
   publisher={Springer-Verlag},
     address={New York},
        date={1979},
      volume={67},
        ISBN={0-387-90424-7},
        note={Translated from the French by Marvin Jay Greenberg},
      review={\MR{554237 (82e:12016)}},
}

\bib{ST18}{article}{
      author={Shankar, Ananth~N.},
      author={Tsimerman, Jacob},
       title={Unlikely intersections in finite characteristic},
        date={2018},
     journal={Forum of Mathematics, Sigma},
      volume={6},
       pages={e13},
}

\end{biblist}
\end{bibdiv}
\end{document}